\documentclass[11pt]{amsart}
\usepackage{latexsym, amsmath, amssymb,amsthm,amsopn,amsfonts}
\usepackage{version}
\usepackage{epsfig,graphics,color,graphicx,graphpap}
\usepackage{amssymb}
\usepackage{todonotes}
\usepackage{listings}
\usepackage{mathrsfs}
\usepackage[citecolor=blue]{hyperref}

\usepackage[framemethod=tikz]{mdframed}
\usepackage{tikz}
\usepackage{url}
\definecolor{mycolor}{rgb}{0.122, 0.435, 0.698}

\newmdenv[innerlinewidth=0.5pt, roundcorner=4pt,linecolor=mycolor,innerleftmargin=6pt,
innerrightmargin=6pt,innertopmargin=6pt,innerbottommargin=6pt]{mybox}
\newcommand{\redsout}{\bgroup\markoverwith{\textcolor{red}{\rule[0.5ex]{2pt}{.4pt}}}\ULon}

\newcommand{\LV}{\left|}
\newcommand{\RV}{\right|}
\newcommand{\LC}{\left(}
\newcommand{\RC}{\right)}

\newcommand{\p}{\partial}

\numberwithin{equation}{section}
\newtheorem{theorem}{Theorem}[section]

\newtheorem{proposition}{Proposition}[section]
\newtheorem{lemma}{Lemma}[section]
\newtheorem{definition}{Definition}[section]

\newtheorem{remark}{Remark}[section]
 
\newcommand{\R}{\mathbb R}

\begin{document}

\author[Lai]{Ru-Yu Lai}
\address{School of Mathematics, University of Minnesota, Minneapolis, MN 55455, USA}
\curraddr{}
\email{rylai@umn.edu }

\author[Zhou]{Hanming Zhou}
\address{Department of Mathematics, University of California Santa Barbara, Santa Barbara, CA 93106-3080, USA}
\curraddr{}
\email{hzhou@math.ucsb.edu}

 
\title[Inverse source problems in transport equations with external forces]{Inverse source problems in transport equations with external forces}
\date{\today}

\begin{abstract}
This paper is concerned with the inverse source problem for the transport equation with external force. 
We show that both direct and inverse problems are uniquely solvable for generic absorption and scattering coefficients. In particular, for inverse problems, generic injectivity and a stability estimate of the source are derived. The analysis employs the Fredholm theorem and the Santalo's formula. 

\end{abstract}

\maketitle
 
\section{Introduction} 
We study the inverse source problem for the transport equation with an external force. 
Let $\Omega$ be an open bounded domain in $\R^n$ for $n\geq 2$ with smooth boundary $\p\Omega$.   
We consider a stationary (i.e. time-independent) transport equation on $\Omega$ for a distribution function $u$ depending on position $x$ and velocity $\theta$: 
\begin{align}\label{ODE:transport}
\theta\cdot \nabla_x u(x,\theta)+ F(x,\theta) \cdot\nabla_\theta u(x,\theta)+ \sigma(x,\theta) u(x,\theta)= K(u)(x,\theta) + f(x),  
\end{align}
where $\sigma$ is the attenuation coefficient, which characterizes the total absorption of particles, and $f$ is the source term. The acceleration $F$ describes the total external force acting on the particles. 
Moreover, the scattering operator $K$ is defined by
\begin{equation}\label{scattering operator}
    K(u)(x,\theta) := \int_{V_x} k(x,\theta,\theta') u(x, \theta')\,d\theta' \quad \mbox{for}\quad \theta\in V_x,
\end{equation}
where the kernel $k(x,\theta,\theta')$ characterizes that the particle with the velocity $\theta'$ is scattered into the velocity $\theta$ at the position $x$. Here $\{V_x: x\in\overline\Omega\}$ is a collection of subsets of $\mathbb R^n\setminus \{0\}$, depending smoothly on $x$, to be defined later.

In this paper, we are interested in the case that the external force $F$ has the following form 
\begin{equation}\label{Lorentz force}
F(x,\theta)=-\nabla_x \varphi(x) + Y(x)\, \theta,
\end{equation}
where $\varphi(x)$ is a smooth scalar potential and $Y(x)$ a smooth $n\times n$ skew-symmetric matrix function depending only on the spatial variable (thus can be viewed as a 1-1 tensor). The term $Y(x)\, \theta$ can be understood as the product of a matrix with a column vector $\theta$. For example, in electromagnetism, $-\nabla_x \varphi$ and $Y$ represent the electric and magnetic fields respectively, and $F$ is the Lorentz force induced by the electromagentic field. In the three-dimensional setting, the degree of freedom of $Y$ is $3$, it is easy to see that $Y \theta=\theta\times B$ for some vector function $B$. Here `$\times$' is the cross-product in $\mathbb R^3$. Therefore $Y$ and $B$ are equivalent in three space dimensions, with $B$ the common notation for the magnetic field.

Notice that the transport operator $\theta\cdot \nabla_x+F(x,\theta)\cdot\nabla_\theta$, with the external force $F$ defined by \eqref{Lorentz force}, induces a Hamiltonian flow with the Hamiltonian function (or energy) $\frac{1}{2}|\theta|^2+\varphi(x)$. We define the sphere bundle over $\Omega$ of energy $\tau>\max  \{ \max_{x\in\overline\Omega}\varphi(x), 0\}$ by
$$
S^\tau\Omega:=\{(x,\theta)\in \Omega\times (\mathbb R^n\setminus\{0\}): \frac{1}{2}|\theta|^2+\varphi(x)=\tau\}
$$
with its boundary $\p S^\tau \Omega$. The fiber of the bundle at $x\in\Omega$ is denoted by $S^\tau_x\Omega$, which is exactly the set $V_x$ in the definition of the scattering operator $K$, see \eqref{scattering operator}. Hence the transport equation \eqref{ODE:transport} stays on the bundle $S^\tau\Omega$.
Let $\p_+ S^\tau\Omega$ and $\p_- S^\tau\Omega$ be the outgoing and incoming boundaries respectively and they are defined by
\begin{align*}
\p_\pm S^\tau\Omega := \{(x,\theta)\in \p S^\tau \Omega:\, x\in \p\Omega,\, \pm n(x)\cdot \theta > 0\},
\end{align*}
where $n(x)$ is the unit outer normal vector at $x\in\p\Omega$.
We also denote
$$S^\tau\Omega^2:=\{(x,\theta, \theta'): x\in \Omega,\; \theta, \theta'\in S_x^\tau \Omega\}.$$
See Section \ref{sec:notations} for more definitions and notations. 

\subsection{Inverse Problems and related results}
Consider the transport equation \eqref{ODE:transport} with trivial condition on the incoming boundary
\begin{align}\label{EQN:transport}
\left\{
\begin{array}{ll}
\theta\cdot \nabla_x u(x,\theta)+ F(x,\theta) \cdot\nabla_\theta u(x,\theta)+ \sigma(x,\theta) u(x,\theta)= K(u)(x,\theta) + f(x) & \hbox{in } S^\tau\Omega,\\
u (x,\theta)= 0 &\hbox{on } \p_- S^\tau\Omega.
\end{array} 
\right.
\end{align}
The inverse problem we consider here is to recover the source term $f$ from boundary measurements $\mathcal{A}:L^2(\Omega)\rightarrow L^2(\p_+S^\tau \Omega, d\xi)$ modeled by
$$
    \mathcal{A}f(x,\theta) := u|_{\p_+S^\tau \Omega}\quad \hbox{for all }(x,\theta)\in \p_+S^\tau \Omega,
$$
where $d\xi(x,\theta)= |\theta\cdot n(x)|d\mu(x) d\theta$ and $d\mu(x)$ is Lebesgue measure on $\p\Omega$. 

Before describing our main results, we briefly introduce some related studies.
When there is no external force in \eqref{EQN:transport}, i.e. $F\equiv 0$ (so $\varphi=c$ for some constant $c$, $Y=0$), the particles travel along straight lines. The injectivity of the inverse source problem in the trivial geometry is known in the case of small scattering kernel $k$ \cite{BalTamasan, FST2020} and generic pairs $(\sigma, k)$ \cite{SU2008}. Recent studies on Riemannian manifolds (non-trivial geometry) can be found in \cite{BalMonard19, Sharafutdinov94, Sharafutdinov96, Sharafutdinov99}.
For inverse source problems for time-dependent transport equations by applying Carleman estimates, see for instance \cite{Yamamoto2016, Lai-L, Machida14, MachidaYamamoto20} and the references therein.

If there is no scattering as well, i.e. $F\equiv 0$ and $k\equiv 0$ in \eqref{scattering operator}, then the inverse source problem is equivalent to the invertibility of the Euclidean attenuated X-ray transform. The corresponding boundary measurement $\mathcal A$ can be expressed as an integral
\begin{align}\label{DEF:Isigma}
     I_\sigma f(x,\theta) |_{\p_+(\Omega\times\mathbb S^{n-1})}:= \int_{-\infty}^0 e^{- \int_s^0 \sigma (x+t\theta, \theta) dt} f(x+s\theta) \,ds \quad  \hbox{for all }(x,\theta)\in \p_+(\Omega\times\mathbb S^{n-1}).
\end{align}
Notice that in this case, $\Omega\times\mathbb S^{n-1}=S^{1/2+c}\Omega$ since the energy is $|\theta|^2/2+\varphi = 1/2+c$. For attenuation coefficient $\sigma$ depending on $x$ only, it is known that $I_\sigma$ is injective with explicit inversion formulas \cite{ABK1998,  BS2004, Monard2018, Natterer2001, Novikov2002}. If the attenuation coefficient $\sigma$ depends on both $x$ and the velocity, i.e. $\sigma=\sigma(x,\theta)$, then injectivity results of $I_\sigma$ can be found in e.g. \cite{FSU, PSUZ2019, Zhou2017}.

When $F\neq 0$, under the influence of the acceleration field, the trajectories associated with the transport operator are not straight lines in general (specifically they are curves), and the velocities along such curves are not constant vectors. Therefore, in this scenario $I_\sigma f$   
corresponds to attenuated X-ray transforms along a family of curves instead of straight lines, see \cite{FSU} and \cite[Appendix]{UV2016}.

\subsection{Main Results}
Throughout this paper, we fix the constant $\tau$ and extend the source $f$ from $\Omega$ to the whole space $\R^n$ by $0$. This ensures that the line integral domain can be taken over the whole real line $\R$ without changing the value.
%

To study inverse problem, let $\Omega_1$ be another open bounded domain in $\mathbb R^n$ so that $\Omega\subset\overline \Omega\subset \Omega_1$. We extend the pair $(\sigma, k)$, the potential $\varphi$ and $Y$ to $\Omega_1$ with the same regularity, and such that $\max_{\overline\Omega_1}\varphi<\tau$ as well. We choose and fix such an extension as a continuous operator in those spaces. Define the operator $\mathcal A_1: L^2(\Omega_1)\to L^2(\p_+S^\tau \Omega_1,d\xi)$ in the same way as $\mathcal A$. In particular, we are interested in those $f\in L^2(\Omega_1)$ supported in $\overline\Omega$, which are equivalent to elements of $L^2(\Omega)$. Therefore, the restriction $\mathcal A_1: L^2(\Omega)\to L^2(\p_+S^\tau \Omega_1,d\xi)$ is also well-defined.

We also impose a convexity condition for the boundary $\p\Omega$ with respect to the underlying dynamical flow (or the external force). Note that the trajectories of the particles satisfy the following Newton's equation
\begin{equation}\label{Newton}
\ddot\gamma=F(\gamma,\dot\gamma)=-\nabla_x\varphi(\gamma)+Y(\gamma)\,\dot\gamma.
\end{equation}
We denote the solution of \eqref{Newton} with initial value $(x,\theta)\in S^\tau\overline\Omega$ by $\gamma_{x,\theta}$. See Section \ref{sec:notations} for more details.
\begin{definition}
Let $x\in\p\Omega$, we say that the boundary $\p\Omega$ is strictly convex at $x$ with respect to the external force $F$ if for any $\theta\in S^\tau_x\overline\Omega$ tangent to $\p\Omega$, there exists $\delta>0$ such that $\gamma_{x,\theta}(t)\notin \overline\Omega$ for $t\in(-\delta,\delta)\setminus\{0\}$. Moreover, the boundary $\p\Omega$ is said to be strictly convex with respect to $F$ if it is strictly convex with respect to $F$ at all $x\in\p\Omega$.
\end{definition}
When $F\equiv 0$, the above definition is consistent with the usual definition of a strictly convex domain. We will assume that both $\p\Omega$ and $\p\Omega_1$ are strictly convex with respect to $F$. Notice that this is not an essential assumption, one can always push away the boundary, and manipulate the extensions of $\varphi$ and $k$, to make it strictly convex with respect to $F$.



We say that the family of curves defined by \eqref{Newton} is {\it non-trapping} if for any $(x,\theta)\in S^\tau\overline\Omega$, the curve $\gamma=\gamma_{x,\theta}$ exits the domain $\Omega$ in finite time in both the forward and backward directions. In other words, the travel time functions $\ell_\pm$, defined in \eqref{def:travel time}, satisfy $|\ell_\pm(x,\theta)|<\infty$ for all $(x,\theta)\in S^\tau \overline\Omega$.

Our main result is about the uniqueness and stability of the inverse source problem for generic $(\sigma,k)$.

\begin{theorem}[Uniqueness and stability estimate]\label{THM:unique and stability}
	Let $\Omega$ be an open bounded domain in $\R^n$ for $n\geq 2$ with smooth strictly convex boundary $\p\Omega$ with respect to $F$. Suppose that the trajectories defined by \eqref{Newton} form a 
	family of non-trapping curves. 
	Let $f\in L^2(\Omega)$. 
	Then there exists an open and dense set $\mathscr{V}$ of pairs
	\begin{align}\label{condition sigma k}
	(\sigma(x, \theta),k(x,\theta,\theta'))\in C^2(S^\tau\overline\Omega) \times C^2(S^\tau\overline\Omega^2)   
	\end{align}
	including a neighborhood of $(0,0)$ such that for each $(\sigma,k)$ in the set $\mathscr{V}$, the direct problem \eqref{EQN:transport} is well-posed. Moreover, 
		suppose that the attenuated X-ray transform $I_{\sigma,F}$ (see \eqref{DEF:IsigmaF}) is injective, and
		$
		k(x,\theta,\theta') = \kappa_1(x,\theta) \kappa_2(x,\theta'),
		$ then the following results hold:
	\begin{enumerate}
		\item the map $\mathcal{A}_1$ is injective on $L^2(\Omega)$;
		\item the stability estimate
		\begin{align}\label{stability}
		\|f\|_{L^2(\Omega)}\leq C\|\mathcal{A}_1^*\mathcal{A}_1 f\|_{H^1(\Omega_1)} 
		\end{align}
		holds for all $f\in L^2(\Omega)$, where $C>0$ is a constant and locally uniform in $(\sigma,k)$.
	\end{enumerate}     
\end{theorem} 

 \begin{remark}
 In the above theorem, we assume that the attenuated ray transform $I_{\sigma, F}$, along a family of curves determined by $F$ (or by $\varphi$ and $Y$), is injective. This does not affect much the generality of our result, as $I_{\sigma, F}$ is injective for generic $\varphi$ and $Y$, see \cite{FSU}. In dimension 3 and higher, $I_{\sigma, F}$ is injective if the family of curves satisfies a global foliation condition \cite[Appendix]{UV2016}.
 \end{remark}
 
	Since the dynamic of particle is impacted by both external force $F$ and the scattering term $K$,
	its corresponding boundary measurement $\mathcal{A}_1$ is of the form $I_{\sigma,F} + L$. Here
	\begin{align}\label{DEF:IsigmaF}
	I_{\sigma,F} f(x,\theta) |_{\p_+ S^\tau\Omega}:= \int_{-\infty}^0 e^{- \int_s^0 \sigma (\gamma_{x,\theta}(t), \dot\gamma_{x,\theta}(t)) dt} f(\gamma_{x,\theta}(s)) \,ds, \qquad  (x,\theta)\in \p_+ S^\tau\Omega 	
	\end{align}
    corresponds to attenuated X-ray transforms along a family of curves $\gamma_{x,\theta}$ with velocities $\dot\gamma_{x,\theta}$ and
	$L$ is a compact operator, contributed from the scattering effect. Therefore, the normal operator $\mathcal A^*_1\mathcal A_1$ is equal to $I^*_{\sigma,F}I_{\sigma,F}+\mathcal L$ with a compact operator $\mathcal L$. See Section \ref{sec:inverse problem} for the detailed definitions of compact operators $L$ and $\mathcal L$.
    Note that for the attenuated X-ray transform $I_{\sigma,F}$, its generic injectivity and stability (similar to Theorem \ref{THM:unique and stability}) are known due to \cite{FSU}, where the result even holds for a general family of curves. Since the operator $\mathcal A^*_1\mathcal A_1$ is a compact perturbation of $I^*_{\sigma,F}I_{\sigma,F}$, we then apply the analytic Fredholm theorem \cite{RS80} to establish the injectivity and stability of the operator $\mathcal A_1$, as well as the well-posedness of the forward problem, for generic pairs $(\sigma, k)$.



The remaining part of the paper is organized as follows. In Section~\ref{sec:forward problem}, we formulate the trajectories of particles under the influence of the external force and introduce notations and preliminary results. We also study the forward problem for the transport equation. 
In Section~\ref{sec:inverse problem}, we discuss the determination of the source. We derive uniqueness and stability estimates for a generic class of absorption and scattering coefficients by solving a Fredholm equation related to the normal operator.

\section{The forward problem}\label{sec:forward problem}
In this section, we start by introducing notations and spaces. We also formulate several lemmas and propositions used later to prove the well-posedness result in Section~\ref{sec:well-posedness}.

\subsection{Notations}\label{sec:notations}
The trajectory $(X, \Theta)$ of the flow, associated to the transport operator $\theta\cdot \nabla_x+F(x,\theta)\cdot\nabla_\theta$, is characterized by the following system: 
\begin{align}\label{ODE}
\left\{
\begin{array}{l}
\dot{X}(s) = \Theta(s) \\
\dot{\Theta}(s)=-\nabla_x\varphi(X(s)) + Y(X(s))\, \Theta\\
X(0) = x\\
\Theta(0)=\theta,
\end{array} 
\right.
\end{align}
with initial condition $(x,\theta)$ at initial time $s=0$. 
We introduce the energy
\begin{align*}
H(x,\theta) = {1\over 2} |\theta|^2 + \varphi(x),
\end{align*}
and note that this energy $H$ is conserved, that is, 
$$
H (X(s), \Theta(s)) =H(x,\theta).
$$
In fact, by taking the derivative with respect to $s$ on the energy $H (X(s), \Theta(s)) = {1\over 2} |\Theta(s)|^2 + \varphi(X(s))$, along any fixed flow,
we have
$$
{d\over ds} H(X(s), \Theta(s)) = \Theta(s)\cdot \dot{\Theta}(s) + \nabla_x \varphi(X(s)) \dot{X}(s) =0
$$
due to \eqref{ODE} and the fact that $Y(x)\,\theta$ is orthogonal to $\theta$ ( that is, $Y(x)\,\theta\cdot \theta=0$). This implies that along the flow the energy $H (X(s), \Theta(s))$ is a constant determined by the given initial data.

Therefore, from now on, we consider the energy level $H$ is a fixed positive constant along each flow, that is,  
\begin{align}\label{DEF:H}
H(x,\theta) = {1\over 2} |\theta|^2 + \varphi(x)\equiv \tau 
\end{align}
for some constant $\tau>0$ and also assume that $\max_{\overline\Omega}\varphi < \tau$.
Then \eqref{DEF:H} yields that for each $x\in\overline\Omega$, its velocity indeed lies on a sphere with the radius depending on $x$ only. We can also express the velocity as a function of $x$ and $v$:
$$
    \theta = p(x) v,  
$$
where $v\in\mathbb{S}^{n-1}$, the unit sphere, and the radius $p(x):=\sqrt{2 (\tau-\varphi(x))}$.
We define the sphere bundle over $\Omega$ with a fixed energy by
$$
S^\tau \Omega := \{(x,\theta) \in \Omega\times (\R^n\setminus\{0\}):\, H(x,\theta) = \tau\}
$$
and its boundary by
$$
\p S^\tau \Omega := \{(x,\theta) \in \p\Omega\times (\R^n\setminus\{0\}) :\, H(x,\theta) = \tau\}.
$$
Moreover, we denote the fiber of the bundle at each point $x\in \Omega$ by 
$$
S_x^\tau\Omega :=\{\theta = p(x)v:\, v\in\mathbb{S}^{n-1}\}.
$$
Recall that the outgoing and incoming boundaries are defined through
\begin{align*}
\p_\pm S^\tau\Omega := \{(x,\theta)\in \p S^\tau \Omega:\, x\in \p\Omega,\, \pm n(x)\cdot \theta > 0\},
\end{align*}
where $n(x)$ is the unit outer normal vector at $x\in\p\Omega$.

For $x\in\overline\Omega$ and $\theta\neq 0$, the curve $X(s)$ satisfying the initial conditions $X(0)=x$ and $\dot{X}(0)=\theta$ is defined on the interval $[\ell_-(x,\theta),\ell_+(x,\theta)]$, where
two travel time functions 
\begin{align}\label{def:travel time}
\ell_\pm: \overline\Omega\times (\R^n\setminus\{0\})\rightarrow \R 
\end{align}
are determined by $X(\ell_\pm(x,\theta))\in \p\Omega$.  
In particular, they satisfy $\pm\ell_\pm(x,\theta)\geq0$ and $\ell_+(x,\theta)|_{\p_+S^\tau\Omega}=\ell_-(x,\theta)|_{\p_-S^\tau\Omega}=0$.


\subsection{Preliminary results}\label{preliminary}
To emphasize the trajectory's dependence on the initial data, from now on we denote solutions to \eqref{ODE} with initial data $(x,\theta)$ by
$$
(\gamma_{x, \theta}(s),\dot\gamma_{x, \theta}(s)).
$$

We define the following two operators
$$
T_1 := \theta\cdot \nabla_x + F(x,\theta)\cdot\nabla_\theta + \sigma(x,\theta),
$$
and 
$$
T:= \theta\cdot \nabla_x + F(x,\theta)\cdot\nabla_\theta + \sigma(x,\theta) - K.
$$
Then $T=T_1-K$. The problem \eqref{EQN:transport} can be rewritten as    
\begin{align}\label{ODE:transport T}
   (T_1 - K)u = f,\qquad u|_{\p_-S^\tau\Omega}=0.
\end{align} 
At this point, we assume that $f$ depends on the velocity as well, i.e. $f=f(x,\theta)$. Applying $T_1^{-1}$ to \eqref{ODE:transport T}, we have
$$
   (Id-T_1^{-1} K )u=T_1^{-1}f,
$$
where $Id$ denotes the identity operator and 
the operator $T_1^{-1}$ is defined by
$$
T_1^{-1} f (x,\theta) := \int_{-\infty}^0 e^{- \int^0_s \sigma (\gamma_{x, \theta}(t), \dot\gamma_{x,\theta}(t)) dt} f(\gamma_{x, \theta}(s),\dot\gamma_{x, \theta}(s)) \,ds \quad \hbox{for all }(x,\theta)\in  S^\tau\overline\Omega.
$$
To simplify the expression, we denote the exponential term in the above integral by
$$W(s,x,\theta):=e^{- \int^0_s \sigma (\gamma_{x, \theta}(t), \dot\gamma_{x,\theta}(t)) dt}.$$
If the operator $Id-T_1^{-1} K$ is invertible, then the problem \eqref{ODE:transport T} is uniquely solvable. This implies that 
there exists a unique solution of the problem \eqref{ODE:transport T} of the form
\begin{align}\label{ODE solution}
u= T^{-1} f= (Id - T_1^{-1}K)^{-1} T_1^{-1} f,
\end{align} 
and, moreover, it can also be written as
$$
u=T_1^{-1}(Id- KT_1^{-1})^{-1} f.
$$
Therefore, the study of the forward problem here is reduced to showing that the operator $Id-KT_1^{-1} $ is invertible. 
For this purpose, we will study $KT_1^{-1}K$ instead since $KT_1^{-1}$ is not compact on $L^2(S^\tau\Omega)$ due to insufficient integrals. This is motivated by the following observation. 
The invertibility of 
$$
    Id-(K T_1^{-1})^2=(Id- KT_1^{-1})(Id+ KT_1^{-1})
$$
implies the invertibility of $Id- K T_1^{-1} $, see \cite{SU2008} or Section~\ref{sec:well-posedness} for detailed discussion.

The compactness of the operator $KT_1^{-1}K$ on $L^2(S^\tau \Omega)$ is established in Proposition~\ref{prop:compact 2}, which is built on the following two lemmas.

	\begin{lemma}\label{lemma:A1A2}
		Suppose that $(\sigma, k)$ satisfy \eqref{condition sigma k}. For any $f\in L^2(S^\tau\Omega)$, the operator $KT_1^{-1}$ is decomposed into 
		$$KT_1^{-1}f=\widehat{K}_s f+\widehat K_rf,$$ where $\widehat{K}_s$ is a singular integral operator defined in \eqref{DEF:A1} and $\widehat{K}_r$ is a smooth integral operator defined in \eqref{DEF:A2}.
		
	\end{lemma}
	\begin{proof}
		Let $\chi: \R \rightarrow [0,1]$ be a smooth cutoff function satisfying $\chi=1$ near $s=0$ and its compact support $\text{supp}(\chi)$ is sufficiently small. We split the operator $KT^{-1}_1$ into the following two operators:
		$$
		[KT_1^{-1}] f =: \widehat{K}_sf  + \widehat{K}_rf,
		$$
     	where we denote
	 	$$
     	\widehat{K}_sf := 	[KT_1^{-1}] (\chi f) \qquad \hbox{ and }\qquad \widehat{K}_rf := 	[KT_1^{-1}] ((1-\chi) f ).
	 	$$
	    Performing the change of variables $\theta'=p(x)v$ with $p(x)>0$ for $v\in\mathbb{S}^{n-1}$ gives that 
		\begin{align*}
		[K  T_1^{-1}] f (x,\theta) &= \int_{S_x^\tau\Omega} k(x,\theta,\theta')  \int_{-\infty}^0 W(s,x,\theta') f(\gamma_{x, \theta'}(s),\dot\gamma_{x, \theta'}(s)) \,ds d\theta'\\
		&=  \int_{\mathbb{S}^{n-1}} k(x,\theta, p(x)v)  \int_{-\infty}^0 W(s,x,p(x)v) f(\gamma_{x, p(x)v}(s),\dot\gamma_{x, p(x)v}(s)) \, p(x)^{n-1}ds  dv.
		\end{align*}

		For any $s_0\neq 0$, the map $(s, \theta')\mapsto  y = \gamma_{x, \theta'}(s)$ is a diffeomorphism from a neighborhood of $(s_0,\theta_0)$ to its image. 
		However, when $s_0=0$, say ($s\in (-\varepsilon,0]$ for small $\varepsilon>0$), the map $(s,\theta')\mapsto  y = \gamma_{x, \theta'}(s)$ has the Jacobian $\tilde J$ vanishing at $s=0$. This implies that after performing the change of variables $(s,\theta')\mapsto y$, the operator $\widehat{K}_r$ has a smooth kernel while $\widehat{K}_s$ is the operator with a singular kernel.
		More precisely, if we perform the change of variables $y=\gamma_{x, pv}(s)$ in $\widehat{K}_rf$, then we obtain the smooth operator
		\begin{align}\label{DEF:A2}
		&\widehat{K}_rf(x,\theta) \notag\\= \,&\int_{\Omega}  k(x,\theta, p v)  (1- \chi(s)) W(s,x,pv) f(y,\dot\gamma_{x, p(x)v}(s))  \tilde{J}^{-1}(x,s,v)|_{s=s(y),v=v(y)}\,  p(x)^{n-1}  dy.
		\end{align}

		To deal with the singular kernel in $\widehat{K}_s$, instead of taking the change of variable as above, we take the following procedures.
		We first define the function
		$$
		m(s,p v; x) := { \gamma_{x, pv}(s) - x \over s}. 
		$$
		Then $m$ is smooth for $|s|<\varepsilon$ and we obtain
		$$
		\gamma_{x,pv}(s) - x = s m(s,pv;x),\qquad m(0,pv; x) = pv ,
		$$
		where we used the Taylor's polynomial of $\gamma_{x,pv}$ at $s=0$ and 
		$ d_s|_{s=0}m(s,p v; x) =\dot\gamma_{x,p v}(0) = pv$. 
		
		Next, we introduce the new variables $(r,w)\in \R\times \mathbb{S}^{n-1}$ defined by 
		$$
		r = s|m(s,p v; w)|,\quad w = {m(s,p v; x) \over |m(s,p v;x)|}\qquad \hbox{ for } s\leq 0.
		$$
		These new variables can be viewed as polar coordinates for $\gamma_{x, pv}(s) - x = rw$ in which we allow $r$ to be negative. Moreover, $(r,w)$ are smooth for $\varepsilon$ small enough. 
		For the change of variables $(s,v)\mapsto (r,w)$, we denote its Jacobian by
		$$
		J(x,s,v) := \text{det}{\p (r,w)\over \p(s,v)} \neq 0. 
		$$ 
		To analyze the behavior of $J$ near $s=0$, a direct computation gives
		$$
		    \p_s r |_{s=0} = |m(0,pv;x)|=p(x),\qquad \p_v r |_{s=0}=0.
		$$
		Moreover, since $m(0,p v;x)= \dot\gamma_{x, p v}(0) = pv$ for $v\in \mathbb{S}^{n-1}$ gives 
		$$w(0,pv;s) = {m(0,p v; x) \over |m(0,p v;x)|} = v,$$
		from Taylor's polynomial, we have
		$$
		\p_v w |_{s=0} = \p_v(v+ \p_sw(0,pv;x) s +\ldots) |_{s=0} =I_n,
		$$
		the $n\times n$ identity matrix.
	    This implies that the Jacobian satisfies
		$$
		J|_{s=0}= p(x)\neq 0, 
		$$ 
		and therefore the map $(s,v)\in \R^-\times\mathbb{S}^{n-1}  \mapsto (r,w) \in \R^- \times\mathbb{S}^{n-1}$ is a local diffeomorphism from $(-\varepsilon,0] \times\mathbb{S}^{n-1}$ to its image provided that $\varepsilon>0$ is sufficiently small. Here we denote $\R^-$ to be the set consisting of all negative points and the zero.
		
	    Note that we can always choose the support of $\chi$ to be sufficiently small. Combining this with the fact that $J$ does not vanish near $s=0$,
		we have
		\begin{align*}
	    \widehat{K}_sf (x,\theta)  
		&=  \int_{\mathbb{S}^{n-1}}  k(x,\theta, p v)  \int_{-\infty}^0 \chi(s)  W(s,x,pv) J^{-1}(x,s,v)  \\
		&\hskip4cm f(x+rw,\dot\gamma_{x, p(x)v}(s)) |_{s=s(x,r,w),v=v(x,r,w)} \, p(x)^{n-1} \, drdw.
		\end{align*}
		Finally, the change of variables $y= x+rw$ in $\widehat{K}_sf$ yields
		\begin{align}\label{DEF:A1}
			\widehat{K}_sf(x,\theta) := \int_{\Omega} {k_1(x,y,\theta) \over {|x-y|^{n-1}}} f(y,\dot\gamma_{x, p(x)v}(s))|_{s=s(x,|y-x| ,{y-x\over |y-x|}),v=v(x,|y-x| ,{y-x\over |y-x|})} \, p(x)^{n-1} dy,
		\end{align}
		where
		\begin{align}\label{DEF:k1}
		k_1 (x,y,\theta) :=  k(x,\theta,pv) \chi(s) W(s,x,pv) J^{-1}(x,s,v)|_{s=s(x,|y-x| ,{y-x\over |y-x|}),v=v(x,|y-x| ,{y-x\over |y-x|})} .
		\end{align}
	\end{proof}
	
	\begin{remark}
	Since $J|_{s=0}=p(x)\neq 0$, 
	$J^{-1}$ is bounded from above and below by some positive constants near $s=0$. Therefore, we have
	$|k_1|\leq C$ for some constant $C>0$.
	\end{remark}
	We now introduce the operator $\Pi$ lifting a function from $\Omega$ to $S^\tau\Omega$, that is, 
	$$
	   [\Pi f](x,\theta) :=  f(x).
	$$

	Then we have the following compact operator.
	\begin{lemma}\label{lemma:compact 1}
	The operator $KT_1^{-1}\Pi : L^2(\Omega) \rightarrow L^2(S^\tau\Omega)$ is compact.
	\end{lemma}
	\begin{proof}
		We restrict $f$ on the $x$ variable only and then Lemma~\ref{lemma:A1A2} leads to 	
		\begin{align*}
		[KT_1^{-1}\Pi]f (x,\theta) =: \widehat{K}_s \Pi f(x,\theta) + \widehat{K}_r\Pi f (x,\theta),
		\end{align*}
		where $\widehat K_r\Pi$ is a smooth operator and 
		\begin{align*}
		\widehat{K}_s\Pi f(x,\theta) := \int_{\Omega} {k_1(x,y,\theta) \over {|x-y|^{n-1}}} f(y) \, p(x)^{n-1} dy,
		\end{align*}
		is a singular integral operator with $k_1$ defined in \eqref{DEF:k1}.
        By the assumptions on $k, \sigma$, the function $|k_1(x,y,\theta)p(x)^{n-1}|$ is bounded from above by some positive constant. Following a similar argument as in the proof of Proposition~\ref{Prop:compact operator 1} implies that $\widehat{K}_s\Pi$ is indeed a compact operator by exhibiting it as a norm limit of compact operators with finite ranks. Therefore, together with the smoothness of the operator $\widehat{K}_r\Pi$, we conclude that $KT_1^{-1}\Pi$ is compact on $L^2(\Omega)$.	 
	\end{proof}

To close this subsection, we show the following result
with the help of Lemma~\ref{lemma:A1A2} and Lemma~\ref{lemma:compact 1}. 
\begin{proposition}\label{prop:compact 2}
	The operator $KT_1^{-1}K : L^2(S^\tau\Omega) \rightarrow L^2(S^\tau\Omega)$ is compact.
\end{proposition}
\begin{proof}
Since $\widehat{K}_r K$ is a smooth operator, it is sufficient to show that $\widehat{K}_s K$ is compact. To achieve this, replacing $f(y,\dot\gamma_{x, p(x)v}(s))$ in $\widehat{K}_sf$ in \eqref{DEF:A1} by 
\begin{align*}
&K(f)(y,\dot\gamma_{x, p(x)v}(s))|_{s=s(x,|y-x| ,{y-x\over |y-x|}),v=v(x,|y-x| ,{y-x\over |y-x|})}\\
=\,& \int_{S_y^\tau\Omega} k(y,\dot\gamma_{x, p(x)v}(s),\theta') |_{s=s(x,|y-x| ,{y-x\over |y-x|}),v=v(x,|y-x| ,{y-x\over |y-x|})} f(y, \theta')\,d\theta'\\
=\,& \int_{\mathbb{S}^{n-1}}  k(y,\dot\gamma_{x, p(x)v}(s),p(y)v')|_{s=s(x,|y-x| ,{y-x\over |y-x|}),v=v(x,|y-x| ,{y-x\over |y-x|})} f(y, p(y)v')\,p(y)^{n-1}\, dv',
\end{align*}
we have
\begin{align*}
   [\widehat{K}_sK f]  (x,\theta)  
  = \int_\Omega  \int_{\mathbb{S}^{n-1}}   { g(x,y,\theta,v')  \over |x-y|^{n-1} } f(y,p(y)v') p(y)^{n-1}  \,  dv'dy.
\end{align*}
Here we denote
$$
g(x,y,\theta,v') :=k_1(x,y,\theta)   k(y,\dot\gamma_{x, p(x)v}(s),p(y)v') |_{s=s(x,|y-x| ,{y-x\over |y-x|}),v=v(x,|y-x| ,{y-x\over |y-x|})}  p(x)^{n-1}  ,
$$
which satisfies  
$$
   |g(x,y,\theta,v') |\leq C,
$$
for some constant $C>0$. 
By Proposition~\ref{Prop:compact operator 1}, we conclude that $\widehat{K}_sK$ is compact. This completes the proof.
\end{proof}

\subsection{The well-posedness result}\label{sec:well-posedness}
Before showing the well-posedness result, we also need the following analytic Fredholm theorem from \cite{RS80}. Let $\mathscr L(\mathscr H)$ denote the set of bounded linear operators on the Hilbert space $\mathscr H$.

\begin{proposition}\cite[Theorem VI.14]{RS80}\label{analytic Fredholm}
Let $D$ be an open connected subset of $\mathbb C$. Let $\mathcal{F}: D\to \mathscr L(\mathscr{H})$ be an analytic operator-valued function such that $\mathcal{F}(\lambda)$ is compact for each $\lambda\in D$. Then one of the following statements holds: 
either 
\begin{enumerate}
		\item[(1)] $(Id-\mathcal{F}(\lambda))^{-1}$ exists for no $\lambda\in D$; or else 
\item[(2)] $(Id-\mathcal{F}(\lambda))^{-1}$ exists for all $\lambda\in D\setminus S$ where $S$ is a discrete subset of $D$. In this case, $(Id-\mathcal{F}(\lambda))^{-1}$ is meromorphic in $D$, analytic in $D\setminus S$, the residues at the poles are finite rank operators.
\end{enumerate}
\end{proposition}

We are now ready to show that the problem under study is well-posed. Note that the following well-posedness result holds for $f=f(x,\theta)$. 

\begin{theorem}[Well-posedness Result]\label{THM:well-posedness} 
	There exists an open and dense set $\mathscr U$ of pairs $(\sigma,k)\in C(S^\tau \overline\Omega) \times  C(S^\tau\overline\Omega^2)$, including a neighborhood of $(0,0)$ so that for each $(\sigma,k)$ in the set, the following statements hold:
	\begin{enumerate}
		\item[(1)] for any $f\in L^2(S^\tau\Omega)$, there exists a unique solution $u\in L^2(S^\tau\Omega)$ to the problem \eqref{EQN:transport};
		\item[(2)] $\mathcal{A}: L^2(S^\tau\Omega)\rightarrow L^2(\p_+S^\tau\Omega,d\xi)$ is a bounded operator.
	\end{enumerate}
\end{theorem}
Here $d\xi(x,\theta):=|n(x)\cdot \theta|d\mu(x) d\theta$ is the measure on $\p_+S^\tau \Omega$ with $d\mu(x)$ the standard measure on $\p\Omega$.
\begin{proof}
	(1) Following a similar argument and notations as in \cite{SU2008}, we fix an arbitrary pair $(\sigma,k)$, and consider 
	$$P(\lambda)=Id-(\lambda K T_1^{-1})^2=(Id-\lambda KT_1^{-1})(Id+\lambda KT_1^{-1})\quad \mbox{in} \quad L^2(S^\tau\Omega),\quad \lambda\in\mathbb C.$$
	Notice that $(\lambda KT_1^{-1})^2=\lambda^2 KT_1^{-1}KT_1^{-1}$ is compact by Proposition~\ref{prop:compact 2} and the resolvent $P(\lambda)^{-1}$ exists for $|\lambda|\ll 1$. By taking $F(\lambda) = (\lambda K T_1^{-1})^2$ and $D=\mathbb C$, we apply Proposition~\ref{analytic Fredholm}   to get that the resolvent $P(\lambda)^{-1}$ exists for all but a discrete set $S\subset \mathbb C$ and it is meromorphic w.r.t. $\lambda$. In particular, this implies that the resolvent
	$$(Id-\lambda KT_1^{-1})^{-1}=(Id+\lambda KT_1^{-1})P(\lambda)^{-1}$$
	exists for $\lambda\in \mathbb C\setminus S$. This can be shown by analytic continuation in $\mathbb C\setminus S$, which is an open connected set. 
    Since $(Id-\lambda T_1^{-1}K)^{-1} T_1^{-1} = T_1^{-1}(Id-\lambda KT_1^{-1})^{-1}$, the operator $T^{-1}$ with $k$ replaced by $\lambda k$ in \eqref{ODE solution} exists and thus
    the solution to \eqref{EQN:transport} 
	$$
	u=(Id-\lambda T_1^{-1}K)^{-1} T_1^{-1} f=T_1^{-1}(Id-\lambda KT_1^{-1})^{-1} f
	$$ 
	exists for any $\lambda\in \mathbb{C}\setminus S$. This means that for $(\sigma,\lambda k)$ $\lambda\in \mathbb{C}\setminus S$, the solution $u$ exists, which yields that the set $\mathscr U$ is dense in $C(S^\tau\overline{\Omega})\times C(S^\tau\overline{\Omega}^2)$. 
	
	To show this set $\mathscr U$ is open, we apply the standard perturbation arguments. We consider $k$ for which $Id - \lambda KT_1^{-1}$ is invertible and $\tilde k$ is a small perturbation of $k$ ($\tilde k$ is close to $k$ in $C(S^\tau\overline{\Omega}^2)$), which leads to that the operator $\widetilde K$ with kernel $\tilde{k}$ is close to $K$. Moreover, we take $\tilde \lambda$ a small perturbation of $\lambda$. Then the invertibility of $Id-\lambda KT_1^{-1}$ gives that
	\begin{align*}
	Id-\tilde \lambda \widetilde KT_1^{-1} & = Id-\lambda KT_1^{-1}  -\lambda (\widetilde K - K)
	T_1^{-1} + (\lambda-\tilde\lambda)\widetilde K T^{-1}_1  \\
	& = (Id-\lambda KT_1^{-1}) [Id-(Id-\lambda KT_1^{-1})^{-1} (\lambda(\widetilde K-K)T^{-1}_1 - (\lambda-\tilde\lambda)\widetilde K T^{-1}_1)]\\
	&=:(Id-\lambda KT_1^{-1}) (Id- K^\flat) .
	\end{align*}
	When $\tilde k$ and $\tilde \lambda$ are sufficiently close to $k$ and $\lambda$, respectively, the norm of the operator $K^\flat$ is going to be very small, which implies that $Id-K^\flat$ is invertible.
	Thus we can conclude that $Id-\tilde \lambda \widetilde KT_1^{-1}$ is invertible when $\tilde k$ is close to $k$ and $\tilde\lambda$ is also close to $\lambda$.
	This shows that the set $\mathscr U$ is open. Hence we complete the proof of part (1).

	
	(2)
	For $f\in L^2(S^\tau\Omega)$, from (1), we have that the transport solution $u=(Id -T^{-1}_1 K)^{-1} T_1^{-1} f$ is in $L^2(S^\tau\Omega)$. Then the generalized attenuated X-ray transform is
	$$
	    \mathcal{A}f= B_+ T^{-1} f = B_+(Id -T^{-1}_1 K)^{-1} T_1^{-1} f=B_+T_1^{-1}(Id - KT^{-1}_1 )^{-1}  f,	
	$$
where we denote 
$$B_+h := h|_{\p_+S^\tau \Omega}$$
and, in particular, $(Id - KT^{-1}_1 )^{-1} f\in L^2(S^\tau\Omega)$. Therefore, it is sufficient to show that $B_+T_1^{-1} : L^2(S^\tau\Omega)\rightarrow L^2(\p_+S^\tau\Omega,d\xi)$ is also bounded. To this end, we apply H\"older inequality and Santalo's formula (see Proposition \ref{Santalo}) to obtain
\begin{align*}
&\hskip.5cm \|B_+T_1^{-1}  f\|^2_{L^2(\p_+S^\tau\Omega,d\xi)} \\
&= \int_{\p_+S^\tau\Omega} |B_+T_1^{-1}  f(x,\theta)|^2 d\xi(x,\theta) \\
&\leq C\int_{\p_+S^\tau\Omega} \LC\int^0_{\ell_-(x,\theta)} |f(\gamma_{x,\theta}(s),\dot\gamma_{x,\theta}(s))| ds\RC^2 d\xi(x,\theta)\\
&\leq C (\max_{\p_+S^\tau\Omega}|\ell_-(x,\theta)|) \int_{\p_+S^\tau\Omega}  \int^0_{\ell_-(x,\theta)} \LV f(\gamma_{x,\theta}(s),\dot\gamma_{x,\theta}(s))\RV^2  ds  d\xi(x,\theta)\\
&\leq  C (\max_{\p_+S^\tau\Omega}|\ell_-(x,\theta)|) \int_{\p_+S^\tau\Omega}    \LC  \int^0_{\ell_-(x,\theta)}   p(\gamma_{x,\theta}(s)) \LV f(\gamma_{x,\theta}(s),\dot\gamma_{x,\theta}(s))\RV^2  ds \RC p(x)^{-1}d\xi(x,\theta)\\
&= C(\max_{\p_+S^\tau\Omega}|\ell_-(x,\theta)|) \int_\Omega \int_{S_x^\tau\Omega} |f(x,\theta)|^2 \,d\theta dx \\
&= C(\max_{\p_+S^\tau\Omega}|\ell_-(x,\theta)|) \|f\|^2_{L^2(S^\tau
\Omega)}. 
\end{align*}
The proof is completed.
\end{proof}


\section{Inverse Source Problem}\label{sec:inverse problem}

In this section, we will study the recovery of the source. For this purpose, we consider a larger domain $\Omega_1\Supset \Omega$ and $\Omega_1$ is also strictly convex and bounded. Moreover, as mentioned in the introduction, we extend $\varphi$, $Y$, $\sigma$ and $k$ to $\Omega_1$ while keeping their regularities and for the source term, we also extend $f$ by zero in $\Omega_1\setminus\Omega$. Then $f$ is supported in $\overline\Omega$. 
In this larger domain $\Omega_1$, we define the corresponding measurement operator $\mathcal{A}_1$ by
$$
    \mathcal{A}_1: f\in L^2(\Omega_1)\rightarrow L^2(\p_+S^\tau\Omega_1,d\xi)
$$
in the same way as the previously defined operator $\mathcal{A}$. To recover $f$, we will take the measurement on $\p\Omega_1$ instead of $\p\Omega$. 

Note that all operators below are also defined in the same way as before, the only difference is that they are now defined in the domain $\Omega_1$. 

Suppose that for $(\sigma, k)\in\mathscr{U}$, the corresponding operator $T_1-K$ is invertible. This is promised by the well-posedness theorem. Then we can express
$$
\mathcal{A}_1f = B_+T_1^{-1}(Id - KT^{-1}_1 )^{-1} \Pi f =: I_{\sigma,F} f + L f,	
$$
where the operator $I_{\sigma,F}$ is defined in \eqref{DEF:IsigmaF}, written as
$I_{\sigma,F}|_{\p_+ S^\tau\Omega_1} = B_+(T_1^{-1} \Pi)$, and $L:L^2(\Omega_1)\rightarrow L^2(\p_+S^\tau\Omega_1,d\xi)$ is defined by
$$
    L:= B_+ (-Id + (Id-T_1^{-1}K)^{-1})T_1^{-1} \Pi.
$$
Moreover, $L$ can be recast as
$$
L = B_+T_1^{-1} (Id-KT_1^{-1})^{-1} KT_1^{-1}\Pi  
$$
and from this expression, one can see that $L$ is a bounded operator since $B_+T_1^{-1}: L^2( S^\tau\Omega_1)\rightarrow L^2(\p_+S^\tau\Omega_1,d\xi)$, $(Id-KT_1^{-1})^{-1} :L^2(S^\tau\Omega_1)\rightarrow L^2(S^\tau\Omega_1)$ and $KT_1^{-1}\Pi: L^2(\Omega_1)\rightarrow L^2(S^\tau\Omega_1)$ are bounded operators.
 
Now with $\mathcal{A}_1^*$, the adjoint of $\mathcal{A}_1$, we decompose the operator $\mathcal{A}_1^*\mathcal{A}_1$ into
\begin{align}\label{DEF:A*A}
    \mathcal{A}_1^*\mathcal{A}_1 = I_{\sigma,F}^* I_{\sigma,F} + (I^*_{\sigma,F} L + L^*I_{\sigma,F} + L^*L) =: I_{\sigma,F}^* I_{\sigma,F} + \mathcal{L}.
\end{align}
Therefore, we can view $\mathcal{A}_1^*\mathcal{A}_1$ as a compact perturbation of $I^*_{\sigma,F} I_{\sigma,F}$, where the perturbation $\mathcal{L}$ is a compact operator on $L^2(\Omega_1)$ as shown in Lemma~\ref{compact pertubation}. To achieve this, we first study the adjoint of $I_{\sigma,F}$.

\subsection{Analysis of $\mathcal{L}$ and a formula for $I^*_{\sigma, F}$}  
We first derive a formula for $I_{\sigma,F}^*$. Notice that the map $\p_+ S^\tau \Omega_1 \times (-\infty,0) \ni (z,\theta,s) \mapsto (x,\theta')\in S^\tau\Omega_1$ given by 
$x = \gamma_{z,\theta}(s)$, and $\theta' = \dot\gamma_{z,\theta}(s)$,
which is a local diffeomorphism. The inverse of the map can be found by solving
$
\gamma_{x,\theta'}(-s) = z, \, \dot\gamma_{x,\theta'}(-s) = \theta
$
and then its corresponding Jacobian is denoted by 
$$
J^b(x,\theta') := \det {\p(z,\theta,s)\over \p(x,\theta')}.
$$
Recall that $W(s,x,\theta):=e^{- \int^0_s \sigma (\gamma_{x, \theta}(t), \dot\gamma_{x,\theta}(t)) dt}.$ Then for $f\in L^2(\Omega_1)$, $g\in L^2(\p_+ S^\tau\Omega_1, d\xi)$, we have
\begin{align*}
\langle I_{\sigma,F} f,g\rangle 
&= \int_{\p_+ S^\tau \Omega_1} \int^0_{-\infty} W(s,z,\theta) f(\gamma_{z,\theta}(s)) \overline{g}(z,\theta) \,ds d\xi(z,\theta)\\ 
&= \int_{\Omega_1} \int_{S_x^\tau \Omega_1} W(x,\theta') f(x)  \overline{g}^\sharp(x,\theta') J^b(x,\theta') \,d\theta' dx\\
&= \int_{\Omega_1} f(x)  \LC \int_{S_x^\tau \Omega_1} W(x,\theta')   \overline{g}^\sharp(x,\theta') J^b(x,\theta') \,d\theta'\RC dx,
\end{align*}
where we denote
$$
g^\sharp(x,\theta') := g(z,\theta),
$$
that is, $g^\sharp$ is extended as a constant along the curve $\gamma_{z,\theta}(s)$, and $W(x,\theta') := W(s,z,\theta)$. 
Therefore, we have the adjoint of $I_{\sigma,F}$ as follows:
$$
[I^*_{\sigma,F} g](x) := \int_{S_x^\tau \Omega_1} \overline W(x,\theta') J^b(x,\theta') g^\sharp(x,\theta') \,d\theta'.
$$
This yields that (see \cite{FSU})
\begin{align*}
&\hskip.5cm[I^*_{\sigma,F} B_+ T^{-1}_1 h](x) \\
&=   \int_{S_x^\tau \Omega_1} \overline W(x,\theta') J^b(x,\theta') \int_{\R} W(\gamma_{x,\theta'}(s), \dot\gamma_{x,\theta'}(s)) h(\gamma_{x,\theta'}(s), \dot\gamma_{x,\theta'}(s))\,ds  \,d\theta'.
\end{align*}
    Therefore, when $h(x)=h(x,\theta)$, by performing similar change of variables as in Lemma~2.1 for $\widehat{K}_s$, we have, modulo a smoothing operator applied to $h$:
	\begin{align*}
	[I^*_{\sigma,F} B_+ T^{-1}_1 \Pi h](x)  =  \int_{\mathbb{S}^{n-1}} \int_{\R} B(x,r,w) h(x+rw) \,drdw,
	\end{align*}
	with $\theta'=p(x) v$ and 
	$$
	B(x,r,w) =\chi(s) J^{-1}(x,s,\theta') \overline W(x, \theta') J^b(x,\theta') W(\gamma_{x,\theta'}(s), \dot\gamma_{x,\theta'}(s))  p(x)^{n-1}|_{s=s(x,r,w),v=v(x,r,w)}.
	$$
We denote $B_{even}(x,r,w) = (B(x,r,w)+B(x,-r,-w))/2$. We can thus integrate over $r\geq 0$ and double the result
\begin{align*}
[I^*_{\sigma,F} B_+ T^{-1}_1 \Pi h](x)  = 2\int_{\mathbb{S}^{n-1}}\int^\infty_0 B_{even}\LC x,r, w\RC h(x+rw) \,drdw.  
\end{align*}
Therefore, the change of variables $y=x+rw$ leads to
\begin{align}\label{DEF:IB+T-1}
[I^*_{\sigma,F} B_+ T^{-1}_1 \Pi h](x)  =  2\int_{\Omega_1}   B_{even}\LC x, |x-y|,{x-y\over|x-y|}\RC { h(y)\over |x-y|^{n-1}}  \,dy.
\end{align}

To analyze the operator $\mathcal L$, we also need the following result regarding weakly singular integral operators.

\begin{proposition}\cite[Proposition 3.4]{SU2008}\label{bounded weakly singular}
    Let $A$ be the operator 
    $$Af(x)=\int \frac{\alpha(x,y,|x-y|,\frac{x-y}{|x-y|})}{|x-y|^{n-1}} f(y)\, dy$$
    with $\alpha(x,y,r,\theta)$ compactly supported in $x,y$. If $\alpha\in C^2$, then $A: L^2\to H^1$ is continuous with a norm not exceeding $C\|\alpha\|_{C^2}$.
\end{proposition}

\begin{lemma}\label{compact pertubation}
		Suppose that 
		$$
		k(x,\theta,\theta') = \kappa_1(x,\theta) \kappa_2(x,\theta').
		$$
The operator $\mathcal{L}$ is compact on $L^2(\Omega_1)$. Moreover, $\p_x \mathcal{L}$ is also compact on $L^2(\Omega_1)$.
\end{lemma}
\begin{proof}
(1) Note that $KT_1^{-1}(Id-KT_1^{-1})^{-1} \Pi =(Id-KT_1^{-1})^{-1} KT_1^{-1}\Pi$, and thus we have
$$
L = B_+T_1^{-1}KT_1^{-1}(Id-KT_1^{-1})^{-1} \Pi = B_+T_1^{-1} (Id-KT_1^{-1})^{-1} KT_1^{-1}\Pi.
$$
To show that $I^*_{\sigma, F} L$ is compact, we recall that $KT_1^{-1}\Pi$ is a compact operator in Lemma~\ref{lemma:compact 1} and, moreover, both $B_+T_1^{-1}$ and $(Id-KT_1^{-1})^{-1}$ are bounded operators. Combining all these together yields that $L$ is compact. Since $I^*_{\sigma, F}$ is a bounded operator mapping from $L^2(\p_+S^\tau\Omega_1,d\xi)$ to $L^2(\Omega_1)$, we thus obtain that $I^*_{\sigma, F} L$ is a compact operator on $L^2(\Omega_1)$.

From the above argument, we see that the operator $L$ is compact, and thus the adjoint operator $L^*$ is compact as well. Therefore, we conclude that both $L^* L$ and $L^* I_{\sigma, F}$ are compact operators. 
This completes the first part of the proof.

(2) To show that $\p_x\mathcal{L}$ is also compact, we also start by studying $\p_x I^*_{\sigma, F} L$.
From the assumption on the kernel, we can write operator $K$ with kernel $\kappa_1\kappa_2$. This gives
$$
    [KT_1^{-1} h](x,\theta) = \kappa_1(\theta) [Gh](x) 
$$
with 
$$
[Gh](x): =[\widehat{G}_r h + \widehat{G}_s h](x),
$$
where $\widehat{G}_r$ is a smooth operator and $\widehat{G}_s$ is a singular integral operator as in Lemma~\ref{lemma:A1A2}, but with the kernel depending only on $x,\theta'$. Then
$\p_x I^*_{\sigma, F} L$ can be expressed as
$$
\p_x I^*_{\sigma,F} L = (\p_x I^*_{\sigma,F} B_+T_1^{-1} \kappa_1\Pi) (G(Id-KT_1^{-1})^{-1} \Pi).
$$ 

Moreover, we have known that $(Id-KT_1^{-1})^{-1}KT_1^{-1}\Pi$ is compact from (1). The compactness of $G\Pi$ follows from the assumption on $k\in C^2$ and Proposition~\ref{bounded weakly singular}. Hence we can derive that 
$$
G(Id-KT_1^{-1})^{-1} \Pi = G(\Pi + (Id-KT_1^{-1})^{-1}KT_1^{-1}\Pi)
$$
is also a compact operator. Now it remains to show that $\p_x I^*_{\sigma,F} B_+T_1^{-1} \kappa_1 \Pi$ is bounded.
From \eqref{DEF:IB+T-1}, we have
\begin{align*}
&[I^*_{\sigma,F} B_+ T^{-1}_1 \kappa_1 \Pi 	h](x) \\
=\,&  2\int_{\Omega_1} \kappa_1(\gamma_{x,pv}(s),\dot\gamma_{x,pv}(s)) |_{ s=s(x,|x-y|,{x-y\over|x-y|}),v=v(x,|x-y|,{x-y\over|x-y|})} \\
&\hskip4cm \times B_{even}\LC x,|x-y|,{x-y\over|x-y|}\RC  {  h(y)\over |x-y|^{n-1}}  \,dy.
\end{align*}
We then have $\p_x I^*_{\sigma,F} B_+T_1^{-1} \kappa_1\Pi: L^2(\Omega_1)\rightarrow L^2(\Omega_1)$ is bounded by Proposition \ref{bounded weakly singular}. Hence $\p_x I^*_{\sigma, F} L$ is compact.

To analyze the operator $\p_x L^* L$, we follow the proof for $\p_x I^*_{\sigma, F} L$ to get
$$\p_x L^* L=(\p_x L^* B_+T_1^{-1}\kappa_1\Pi)(G(Id-KT_1^{-1})^{-1}\Pi).$$
Recall that $G(Id-KT_1^{-1})^{-1}\Pi$ is compact, therefore it suffices to show that $\p_x L^* B_+T_1^{-1}\kappa_1\Pi: L^2(\Omega_1)\to L^2(\Omega_1)$ is bounded. Note that
\begin{align*}
    \p_x L^* B_+T_1^{-1}\kappa_1\Pi & =\p_x (B_+T_1^{-1} (Id-KT_1^{-1})^{-1} KT_1^{-1}\Pi)^* B_+T_1^{-1}\kappa_1\Pi\\
    & =\p_x (KT_1^{-1}\Pi)^*(B_+ T_1^{-1}(Id-KT_1^{-1})^{-1})^* B_+T_1^{-1}\kappa_1\Pi.
\end{align*}
Since $B_+T_1^{-1}$ and $(Id-K T_1^{-1})^{-1}$ are bounded, it remains to study the operator $\p_x(K T_1^{-1}\Pi)^*$. Recall the kernel of $K T_1^{-1}\Pi$ in the proof of Lemma \ref{lemma:compact 1} and thus a similar expression holds for the adjoint $(K T_1^{-1}\Pi)^*$. Then we apply Proposition \ref{bounded weakly singular} again to conclude that $\p_x(K T_1^{-1}\Pi)^*$ is bounded, and the compactness of $\p_x L^* L$ follows.

To show that $\p_x L^* I_{\sigma, F}$ is compact, we recall that $I_{\sigma, F}|_{\p_+S^\tau \Omega_1}=B_+T^{-1}_1\Pi$, thus
\begin{align*}
  \p_x L^* I_{\sigma, F} & =\p_x L^* B_+ T^{-1}_1 \Pi  \\
  & =\p_x (KT_1^{-1}\Pi)^* (B_+ T_1^{-1}(Id-KT_1^{-1})^{-1})^* B_+ T_1^{-1}\Pi.
\end{align*}
Now the compactness of $\p_x L^* I_{\sigma, F}$ follows immediately from the proof for $\p_x L^* L$ by taking $\kappa_1\equiv 1$.


\end{proof}


\subsection{Proof of Theorem~\ref{THM:unique and stability}}
\begin{proof}[Proof of Theorem~\ref{THM:unique and stability}]
Since real analytic functions are dense in $C^2(S^\tau\overline\Omega_1)$, we can now assume that   $\sigma$ is real analytic. By \cite[Proposition~2]{FSU}, there exists a parametric $Q$ of order $1$ to the elliptic pseudodifferential operator ($\Psi DO$) $I^*_{\sigma, F} I_{\sigma, F}$ in $\Omega_1$. We can also restrict the image of $Q$ to $L^2(\Omega)$ and thus after this restriction, we can view $Q:H^1(\Omega_1) \rightarrow L^2(\Omega)$. Taking any $f\in L^2(\Omega)$, extended by zero outside $\Omega$, so that such $f$ is in $L^2(\Omega_1)$,  
we have
$$
    Q I^*_{\sigma, F} I_{\sigma, F} f = f +R f,
$$
where the operator $R $ is of order $-1$ with a smooth kernel.
From \eqref{DEF:A*A}, applying $Q$ to $\mathcal{A}_1^*\mathcal{A}_1$ yields that
\begin{equation}\label{fredholm equation}
Q \mathcal{A}_1^*\mathcal{A}_1  f = f + \widetilde R f,\qquad \widetilde R := R + Q\mathcal{L}.
\end{equation}
By Lemma \ref{compact pertubation}, we have that $\p_x \mathcal{L}:L^2(\Omega_1)\rightarrow L^2(\Omega_1)$ is a compact operator and also $\mathcal{L} f\in H^1(\Omega_1)$. This implies that $Q\mathcal{L}:L^2(\Omega_1)\rightarrow L^2(\Omega)$ is compact since $Q$ is a bounded operator. Together with the smooth operator $R$, we deduce that $\widetilde R: L^2(\Omega)\to L^2(\Omega)$ is compact as well.


{\bf (1) Uniqueness.}
The solvability of the Fredholm equation $Q \mathcal{A}_1^*\mathcal{A}_1  f = f + \widetilde R f$ is equivalent to the invertibility of $Q\mathcal{A}_1^*\mathcal{A}_1$, which then implies that $\mathcal{A}_1^*\mathcal{A}_1$ is injective.
Moreover, since the injectivity of $\mathcal{A}_1^*\mathcal{A}_1$ implies the injectivity of $\mathcal{A}_1$. Therefore, to show the statement (1), it is sufficient to show that $Q \mathcal{A}_1^*\mathcal{A}_1$ is invertible.

To this end, we first fix a real analytic $\sigma$ in $S^\tau\overline{\Omega}_1$. By \cite[Theorem 1 and 2]{FSU}, the operator $I^*_{\sigma,F} I_{\sigma,F}$ is injective and this injectivitity holds as well for small enough $C^1$ perturbations of $\sigma$. Also, due to Theorem~\ref{THM:well-posedness}, there exists an open dense set $\mathscr U$ such that the well-posedness holds in $\overline{\Omega}_1$ for these $(\sigma, k)$ in the set $\mathscr U$.
We consider the operator $\mathcal{A}_1$ with $(\sigma,\lambda k)$ with $\lambda$ in some complex neighborhood $\mathcal{C}$ of $[0,1]$.


Next we show that the operator $\widetilde R=R  + Q\mathcal{L}$  depends meromorphically on $\lambda\in \mathcal{C}$. To see this, since $L$ is meromorphic by the proof of Theorem \ref{THM:well-posedness},  $\mathcal{L}$ is
a meromorphic function of $\lambda$. This implies that the operator $\widetilde R(\lambda)$ is also a meromorphic function of $\lambda$. Also the proof of Theorem \ref{THM:well-posedness} yields that the operator $\widetilde R(\lambda)$  exists for all but a discrete set $S\subset \mathbb C$. In particular, $\mathcal C\setminus S$ is an open connected subset of $\mathbb C$ as well.  

Moreover, we show that when $\lambda=0$, $Id+\widetilde R(0)$ is invertible.
Note that when $\lambda=0$ that is $(\sigma, 0)$, we have $\mathcal{L}=0$ implying $\widetilde R=R$. Since $R $ is compact, the operator $Id+R $ has a finite dimensional kernel. By following the argument in the proof of \cite[Proposition 4, Theorem 2]{SU4}, we can add a finite rank operator to $Q$, and then the operator $Id+R$ can be arranged to be injective. As a result, $Id+R $ is invertible. 

Combining these facts together, the analytic Fredholm theorem (Proposition~\ref{analytic Fredholm} by letting $D=\mathcal C\setminus S$) implies that 
$$Q\mathcal A^*_1\mathcal A_1=Id+\widetilde R(\lambda) \qquad \hbox{ is invertible for all $\lambda\in\mathcal C$} $$
with the possible exception of a (larger) discrete set $S'\subset \mathcal C$. 
Since $\lambda=1$ is an accumulation point of $\mathcal C\setminus S'$, we further deduce that $\mathcal A_1$ is injective for $(\sigma,\lambda k)$ for those $\lambda\in \mathcal C\setminus S'$ are close to $1$. 
Finally we can conclude that there is a dense set $\mathscr V$ of pairs $(\sigma,k)$ in $\mathscr{U}$ such that $\mathcal A_1$ is injective.

{\bf (2) Stability.}  Besides the injectivity, we can derive a stability estimate in terms of the normal operator $\mathcal A_1^*\mathcal A_1$. If $(\sigma,k)\in \mathscr V$, then $Id+\widetilde R$ is invertible on $L^2(\Omega)$, the Fredholm equation \eqref{fredholm equation} implies that
\begin{equation}\label{stability 2}
\|f\|_{L^2(\Omega)}=\|(Id+\widetilde R)^{-1}Q\mathcal A_1^*\mathcal A_1 f\|_{L^2(\Omega)}\leq C \|\mathcal A_1^*\mathcal A_1 f\|_{H^1(\Omega_1)}.
\end{equation}
Moreover, similar to the proof of Theorem \ref{THM:well-posedness}, the resolvent $(Id+\widetilde R)^{-1}$ depends continuously on $(\sigma,k)$, so is $Q$. 
Therefore, applying perturbation arguments yields that the set $\mathscr V$ is open and the constant $C$ in \eqref{stability 2} is locally uniform in $(\sigma,k)$.
This completes the proof.
\end{proof} 

\begin{remark} There is an alternative proof for the stability by estimating $$\|f\|_{L^2(\Omega)}\lesssim\|\mathcal A_1^*\mathcal A_1 f\|_{H^1(\Omega_1)}+\|\widetilde Rf\|_{L^2(\Omega)}.$$
Since $\mathcal A_1^*\mathcal A_1$ is injective for $(\sigma,k)\in \mathscr V$, and $\widetilde R: L^2(\Omega)\to L^2(\Omega)$ is compact, by \cite[Proposition V.3.1]{Taylor1981}, we obtain the stability estimate \eqref{stability}.
\end{remark}


\appendix
\section{Santalo's formula}
The Santalo's formula is an integral identity that turns an integral over the interior region to the boundary and vice versa. The proof of Santalo's formula relies on the Liouville theorem, which states that the volume is preserved under Hamiltonian flows. For example, the flows associated to divergence free vector fields are volume preserving. 
It is known that a geodesic flow preserves the Liouville measure $d\Sigma^{2n-1}$, specified later, and this formula was shown in \cite{SharafutdinovNote}. The formula also holds for magnetic systems, see \cite{DPSU07}. 

In this section, we prove an analogue of the Santalo's formula for a flow under the influence of external force. Unlike the previous results with unit velocity with respect to (w.r.t.) the metric, the flow has varying velocity $\theta$ depending on the space variable. This explains the presence of the extra terms in \eqref{id:santalo}.

The main idea of the proof of Proposition~\ref{Santalo} is to transform the flow into a magnetic geodesic under a conformally Euclidean metric $g$ so that the new flow has unit speed $\eta$, an unit vector w.r.t. the metric $g$. 
Then we perform the change of variables $\theta\mapsto \eta$. By applying the Santalo's formula for magnetic system \cite{DPSU07}, the integral over the interior region is transformed to the integral over the boundary. Finally, we apply the change of variable again to obtain the desired formula.

To simplify notations, we denote
$$
P(x):=2(\tau-\varphi(x)).
$$
Recall that $p(x) = \sqrt{2(\tau-\varphi(x))}$ and the backward exit time $\ell_-$ is defined in Section~\ref{sec:notations}.

\begin{proposition}[Santalo's formula]\label{Santalo}
	Suppose that $\Omega$ is a bounded convex domain in $\R^n$ with smooth boundary. For any $f\in L^1(S^\tau\Omega)$, one has
	\begin{align}\label{id:santalo}
	\int_\Omega \int_{S_x^\tau\Omega}f(x,\theta) \,d\theta dx = \int_{\p_+ S^\tau\Omega} \LC\int^0_{\ell_-(x,\theta)} P(\gamma_{x,\theta}(s))^{1\over 2}f(\gamma_{x,\theta}(s), \dot\gamma_{x,\theta}(s)) \,ds \RC P(x)^{{-1\over 2}}\, d\xi(x,\theta),
	\end{align}
	where $d\xi(x,\theta)= \langle\theta, n(x)\rangle_e\, d\theta d\mu(x)$ and $d\mu(x)$ is Lebesgue measure on $\p\Omega$ and $n(x)$ is the unit outer normal to $\p\Omega$ w.r.t. the Eucliean metric $e$.
\end{proposition} 
\begin{proof}
Since $\frac{1}{2}|\theta|^2+\varphi=\tau$, we consider the conformal Euclidean metric $g=2(\tau-\varphi)e$, where $e$ is the Euclidean metric. If $\gamma(t)$ is a curve for Hamiltonian with external force at energy level $\tau$, let 
$$s(t)=\int_0^t 2(\tau-\varphi(\gamma(a)))\, da,$$
then the Maupertuis’ principle implies that $\zeta(s):=\gamma(t(s))$ is a unit speed magnetic geodesic w.r.t. the metric $g$, see \cite[Proposition 1]{AssylbekovZhou15}. Let $S^g\Omega$ be the unit sphere bundle w.r.t. $g$, if $\eta\in S_x^g\Omega$, then $\theta=2(\tau-\varphi(x))\eta\in S_x^\tau\Omega$. We also denote
$$
\p_+ S^g\Omega :=\{(x,\eta) :\, x\in\p\Omega,\, \eta \in S^g_x\Omega,\, g(\eta,n_g)>0\},
$$
where $n_g(x)$ is the unit outer normal w.r.t metric $g$ at $x\in\p\Omega$.
Note that 
$$
ds= P(x)dt,\quad \dot{\zeta}(s) = \dot{\gamma}(t(s))P^{-1}(\gamma(t(s))) = \eta.
$$
We denote the Riemannian volume form on $\Omega$ by
$$
dV^n(x)=\sqrt{\text{det}g(x)}\,dx=P^{n/2}(x) dx
$$ 
and the Riemannian volume form on $\p\Omega$ by  
$
dV^{n-1}(x)$ 
satisfying 
\begin{align}\label{DEF:Vn-1}
dV^{n-1}(x)=P^{(n-1)/2}(x) d\mu(x),
\end{align}
where $d\mu(x)$ is the surface measure on $\p\Omega$ w.r.t. the Euclidean metric.  
Let $dw_x(\eta)$ denote the angle measure on the unit sphere $S^g_x\Omega$ with 
$$
\int_{S^g_x\Omega} dw_x(\eta)=\omega_{n-1},
$$
where $\omega_{n-1}$ is the area of the unit sphere (w.r.t. the Euclidean metric) in $\R^n$. In local coordinates, 
\begin{align}\label{def:dw}
 dw_x(\eta)=\sqrt{\mbox{det}\, g(x)}\,\sum_{i} (-1)^{i-1}\eta^i d\eta^1\wedge \cdots \wedge \widehat{d\eta^i}\wedge \cdots\wedge d\eta^n.
\end{align} 
Here $\widehat{d\eta^i}$ designates that the term $d\eta^i$ is omitted. Note that $\theta=P(x)\eta$ yields that
\begin{align}\label{dtheta}
d\theta & =\sum_i (-1)^{i-1}{\theta^i\over |\theta|} d\theta^1\wedge \cdots \wedge \widehat{d\theta^i}\wedge\cdots\wedge d\theta^n \notag\\
& =\sum_i (-1)^{i-1}{\theta^i\over p(x)} d\theta^1\wedge \cdots \wedge \widehat{d\theta^i}\wedge\cdots\wedge d\theta^n \notag\\
& = P^{n-{1\over 2}}(x) \sum_{i} (-1)^{i-1}\eta^i d\eta^1\wedge \cdots \wedge \widehat{d\eta^i}\wedge \cdots\wedge d\eta^n
= P^{n-1\over 2} dw_x(\eta).
\end{align}
Together with $dV^n(x)=P^{n/2}(x) dx$, we have 
$$dxd\theta = P^{-n\over 2}(x)dV^n(x)P^{n-1\over 2} dw_x(\eta)= P^{-1\over 2}(x) dw_x(\eta) \wedge  dV^n(x)=P^{-1\over 2}(x)d\Sigma^{2n-1},$$
where $d\Sigma^{2n-1}(x,\theta)=dw_x(\eta) \wedge  dV^n(x)$ is the Liouville measure. 
This thus implies that
\begin{align*}
\int_{S^\tau \Omega}  f(x,\theta)\, dxd\theta 
&=\int_{\Omega}\int_{S^g_x\Omega}   P^{-1\over 2}(x) f(x,P\eta)\,  d\Sigma^{2n-1}(x,\eta) \\
&=\int_{\p_+ S^g\Omega}  \int^0_{\ell_-(x,\eta)}  P^{-1\over 2}(\zeta(s)) f(\zeta(s),P\dot\zeta(s))\, ds\,  \langle \eta,n_g(x)\rangle_g d\Sigma^{2n-2}(x,\eta) \\
&=: \mathcal{I},
\end{align*}
where in the second identity above, we apply Santal\'o's formula for a compact Riemannian manifold with unit sphere bundle in \cite[Lemma A.8]{DPSU07}. Here we denote $\langle \eta,n_g(x)\rangle_g := g(\eta,n_g(x))$ and, moreover,
$$d\Sigma^{2n-2}(x,\eta)= dw_x(\eta) \wedge dV^{n-1} (x).$$
We now change the integral $\mathcal{I}$ back to the Euclidean metric. By using \eqref{dtheta} and \eqref{DEF:Vn-1}, we get
$$d\Sigma^{2n-2}= P^{-(n-1)\over 2} d\theta  P^{(n-1)\over 2}d\mu(x) = d\theta d\mu(x).$$
Hence
\begin{align*} 
\mathcal{I}
&= \int_{\p_+S^\tau\Omega}   \int^0_{\ell_-(x,\theta)}  P^{-1\over 2}(\gamma(t))f(\gamma(t), \dot\gamma(t))\, P(\gamma(t)) dt\,  P \langle P^{-1}\theta, P^{-1/2} n(x)\rangle_e\,  d\theta d\mu(x)\\
& = \int_{\p_+S^\tau\Omega}   \int^0_{\ell_-(x,\theta)} P(\gamma(t))^{1\over 2} f(\gamma(t), \dot\gamma(t))\, dt\,P(x)^{-1\over 2} \langle \theta, n(x)\rangle_e\,  d\theta  d\mu(x).
\end{align*} 
This completes the proof of the Santalo's formula.

\end{proof}

\section{Weakly singular operators}
 
\begin{proposition}[Compact operators]\label{Prop:compact operator}
	Suppose that $\beta\in L^2(S^\tau \Omega\times S^\tau \Omega)$. The operator 
	\begin{align}\label{DEF:compact operator}
    \Phi f(x,\theta) = \int\int_{S^\tau \Omega}  \beta(x,y,\theta,\theta') f(y,\theta')\,dyd\theta'
	\end{align}
	is compact on $L^2(S^\tau \Omega)$.
\end{proposition}
\begin{proof}
First we show that $\Phi$ is well-defined on $L^2(S^\tau\Omega)$ and is bounded by $\|\beta\|_{L^2(S^\tau \Omega\times S^\tau \Omega)}$.
By H\"older inequality, we have
$$
|\Phi f(x,\theta)|\leq \LC\int\int_{S^\tau \Omega} | \beta(x,y,\theta,\theta')|^2 \,dyd\theta'\RC^{1/2}\|f\|_{L^2(S^\tau\Omega)}.
$$
Then 
\begin{align*}
    \|\Phi f\|_{L^2(S^\tau\Omega)}^2
    &=\int \int_{S^\tau \Omega}|\Phi f(x,\theta)|^2 \,dxd\theta \\
    &\leq \|f\|_{L^2(S^\tau\Omega)}^2  \int\int_{S^\tau \Omega}\int\int_{S^\tau \Omega} | \beta(x,y,\theta,\theta')|^2 \,dyd\theta'dxd\theta \\
    &= \|f\|_{L^2(S^\tau\Omega)}^2\|\beta\|_{L^2(S^\tau \Omega\times S^\tau \Omega)}^2,
\end{align*}
which implies the operator norm of $\Phi$, denoted by $\|\Phi\|_*$, is bounded by $\|\beta\|_{L^2(S^\tau \Omega\times S^\tau \Omega)}$.

Let $\{\phi_k\}_{k=1}^\infty$ be orthonormal bases for $L^2(S^\tau \Omega)$.
Fubini's theorem gives that if $\psi_{ij} (x,y,\theta,\theta')= \phi_i(x,\theta)\phi_j(y,\theta')$, then $\psi_{ij}$ is an orthonormal basis for $L^2(S^\tau \Omega\times S^\tau \Omega)$. For $m=1,2,\ldots$, let 
$$
\beta_m(x,y,\theta,\theta') := \sum_{i+j\leq m} a_{ij} \psi_{ij}(x,y,\theta,\theta'),\quad a_{ij}=\left< \beta, \psi_{ij}\right>_{L^2(S^\tau\Omega\times S^\tau\Omega)} 
$$
and define the operator 
$$
	\Phi_m f(x,\theta) := \int\int_{S^\tau\Omega}  \beta_m (x,y,\theta,\theta') f(y,\theta')\,dyd\theta'.
$$
Then $\Phi_m$ has finite rank and thus $\Phi_m$ is a compact operator.

On the other hand, 
$$
	\|\beta -\beta_m\|^2_{L^2(S^\tau \Omega\times S^\tau \Omega)} = \sum_{i+j>m} |a_{ij}|^2\rightarrow 0 \qquad\hbox{ as }m\rightarrow \infty,
	$$
as a result, 
  $$\|\Phi-\Phi_m\|_*\leq \|\beta-\beta_m\|_{L^2(S^\tau \Omega\times S^\tau \Omega)} \rightarrow 0 \qquad\hbox{ as }m\rightarrow \infty.
  $$
Therefore, $\Phi$ is the norm limit of operators of finite rank. This implies that $\Phi$ is a compact operator. 
\end{proof}

\begin{proposition}\label{Prop:compact operator 1}
	Suppose that $\Omega$ is a bounded domain in $\R^n$. Suppose that $\beta$ satisfies
	$$
	|\beta(x,y,\theta,\theta')|\leq C_1|x-y|^{1-n},\qquad  (x,y,\theta,\theta')\in S^\tau \Omega\times S^\tau \Omega
	$$
	for some positive constant $C_1$.
	Then the linear operator $\Phi:L^2(S^\tau \Omega)\rightarrow L^2(S^\tau \Omega)$ of the form
	\begin{align}\label{DEF: operator 1}
	\Phi f(x,\theta) = \int\int_{S^\tau \Omega}  \beta(x,y,\theta,\theta') f(y,\theta')\,dyd\theta'.
	\end{align}
	is compact on $L^2(S^\tau \Omega)$.
\end{proposition}
\begin{proof}
	For any positive integer $m$, we first define the truncated kernels $\hat\beta_m(x,y,\theta,\theta')$ defined by
	$$
	\hat\beta_m(x,y,\theta,\theta') := \left\{
	\begin{array}{ll}
	\beta(x,y,\theta,\theta') & \hbox{if }|x-y|\geq 1/m,\\
	0 & \hbox{otherwise}.\\
	\end{array}\right.
	$$
	Then $\hat\beta_m$ is in $L^2(S^\tau \Omega\times S^\tau\Omega)$.  
	We define the operator
	$$
	\hat\Phi_m f(x,\theta) := \int\int_{S^\tau\Omega}  \hat\beta_m(x,y,\theta,\theta') f(y,\theta')\,dyd\theta'
	$$
	for all positive integer $m$. By Proposition~\ref{Prop:compact operator}, $\hat\Phi_m$ is compact.
	
	Next we will show that $\|\Phi-\hat\Phi_m\|_* \rightarrow 0$ as $m\rightarrow \infty$.
	To establish this, we denote
	$$
	\varepsilon_m(x,\theta):=\int\int_{S^\tau\Omega}  
	|\beta(x,y,\theta,\theta')-\hat\beta_m(x,y,\theta,\theta')|\,dyd\theta' 
	$$
	and obtain that
	\begin{align*}
	|\varepsilon_m(x,\theta)|&\leq \int\int_{|x-y|<1/m}  | \beta (x,y,\theta,\theta')|\,dyd\theta'\\
	&\leq \int\int_{|x-y|<1/m}  {C_1\over|x-y|^{n-1}}\,dyd\theta'\\
	&\leq Cm^{-1},
	\end{align*}
	where the constant $C$ depends only on $n, \Omega,\tau$ and the function $p$. 
	Then $\varepsilon_m \rightarrow 0 $ when $m\rightarrow \infty$. Note that this convergence is independent of $x$ and $\theta$.

	For every $f\in L^2(S^\tau\Omega)$, we estimate
	\begin{align*}
	&\hskip.5cm\|(\Phi -\hat\Phi_m) f\|_{L^2(S^\tau\Omega)}^2 \\
	&\leq \int\int_{S^\tau\Omega}\LV\int\int_{S^\tau\Omega}  \LC \beta(x,y,\theta,\theta')-\hat\beta_m(x,y,\theta,\theta')\RC  f(y, \theta')  \,dyd\theta'\RV^2 \,dxd\theta\\ 
	&\leq \int\int_{S^\tau\Omega}\LC\int\int_{S^\tau\Omega}  
	|\beta(x,y,\theta,\theta)-\hat\beta_m(x,y,\theta,\theta')|\,dyd\theta'\RC    \\
	&\hskip2cm\LC\int\int_{S^\tau\Omega}  |\beta(x,y,\theta,\theta')-\hat\beta_m(x,y,\theta,\theta')||f(y, \theta')|^2 \,dyd\theta' \RC \,dxd\theta\\
	&\leq C m^{-1}\int\int_{S^\tau\Omega}   
	\LC\int\int_{S^\tau\Omega}  |\beta(x,y,\theta,\theta')-\hat\beta_m(x,y,\theta,\theta')||f(y, \theta')|^2 \,dyd\theta' \RC \  \,dxd\theta\\
    &\leq C m^{-2}\int\int_{S^\tau\Omega}   
	 |f(y, \theta')|^2 \,dyd\theta'  \\
	&\leq C m^{-2} \|f\|_{L^2(S^\tau\Omega)}^2,
	\end{align*}
	which leads to $\|\Phi -\hat\Phi_m\|_*\rightarrow 0$ as $m\rightarrow \infty$. Finally, we conclude that $\Phi$ is a compact operator since it is the limit of compact operators $\hat\Phi_m$.
\end{proof}

\section*{Acknowledgments}
This work is partially supported by the National Science Foundation through grants DMS-2006731.

\bibliographystyle{plain}
\bibliography{externalforce}

\end{document}